\documentclass{amsart}
\date{\today}
\usepackage{graphicx, color, enumerate}
\title{Minimal Darboux transformations}
\author{U.~Hertrich-Jeromin}
\address[Udo Hertrich-Jeromin]{%
   Vienna University of Technology,
   Wiedner Hauptstra\ss{}e 8--10/104, A-1040 Vienna, Austria
}
\email{udo.hertrich-jeromin@tuwien.ac.at}
\author{A.~Honda}
\address[Atsufumi Honda]{%
   National Institute of Technology, Miyakonojo College, 
   Yoshio, Miyakonojo 885-8567, Japan
}
\email{atsufumi@cc.miyakonojo-nct.ac.jp}
\thanks{%
This work has been partially supported by
the Austrian Science Fund (FWF) and
the Japan Society for the Promotion of Science (JSPS)
through the FWF/JSPS Joint Project grant I1671-N26
``Transformations and Singularities''.
}
\subjclass[2010]{%
Primary
53A10,
37K35; 
Secondary
53C42, 
53A30,
37K25,
34M45
}
\keywords{%
 Minimal surface,
 Darboux transformation,
 Christoffel transformation,
 Goursat transformation,
 Bianchi permutability,
 Riccati equation,
 flat front,
 curved flat,
 hyperbolic geometry.
}
\usepackage{amsthm}
\theoremstyle{plain}
 \newtheorem{theorem}{Theorem}[section]

 \newtheorem*{fact*}{Fact}
 \newtheorem{lemma}[theorem]{Lemma}
 \newtheorem{corollary}[theorem]{Corollary}
 \theoremstyle{remark}
 
 \newtheorem{remark}[theorem]{Remark}
 \newtheorem*{acknowledgements}{Acknowledgements}
 
\numberwithin{equation}{section}

\newcommand{\R}{\boldsymbol{R}}
\newcommand{\C}{\boldsymbol{C}}

\newcommand{\bmath}[1]{\mbox{\boldmath $#1$}}
\renewcommand{\H}{{\bmath{H}}}

\newcommand{\inner}[2]{\left\langle{#1},{#2}\right\rangle}

\renewcommand{\Re}{\operatorname{Re}}
\renewcommand{\Im}{\operatorname{Im}}

\begin{document}
\begin{abstract}
We derive a permutability theorem for the
Christoffel, Goursat and Darboux transformations
of isothermic surfaces.
As a consequence we obtain a simple proof of a relation
between Darboux pairs of minimal surfaces in Euclidean
space, curved flats in the $2$-sphere and flat fronts
in hyperbolic space.
\end{abstract}
\maketitle

\section*{Introduction}

Recently, Mart\'{\i}nez-Roitman-Tenenblat \cite{MRT}
obtained a formula for the Weierstrass data of minimal
surfaces in Euclidean $3$-space $\R^3$ that are related
by a Ribaucour transformation \cite[Cor 3.4]{MRT},
which was used to analyze the asymptotic behavior and
the umbilic points of minimal surfaces
obtained by Ribaucour transformations.
To derive this formula,
the authors use the \emph{cyclographic projection}, 
a correspondence between (oriented) spheres in $\R^3$
and the Minkowski $4$-space $\R^4_1$.
In this way, they obtain a relation between
flat fronts in hyperbolic $3$-space
and minimal surface in $\R^3$.

We aim to provide a simpler, virtually computation-free
account of these relations,
based on a permutability theorem for isothermic surfaces.

In \cite{hj97}, see also \cite[Sect 5.3]{IMDG},
the first author extended the Goursat transformation
---
classically known for minimal surfaces,
 see \cite{Goursat1887} 
---
to the class of isothermic surfaces and argued that
the classical Enneper-Weierstrass representation can
be regarded as a special case of this ``extended''
Goursat transformation,
see \cite[\S5.3.12]{IMDG}.
As our main Theorem \ref{thm:main},
we derive a new \emph{permutability theorem} for the
transformations of isothermic surfaces, involving this
Goursat transformation.

The aforementioned relation between
the Weierstrass data of minimal surfaces related by
a Darboux tranformation is then obtained as
a degenerate case of our permutability theorem,
see Corollary \ref{fact:to_prove}.

For the relation to flat fronts in hyperbolic space
we draw on the relation between Darboux pairs of
isothermic surfaces and curved flats in the
space of point pairs in the conformal $3$-sphere,
see \cite{bjpp97} or \cite[\S3.3.2]{IMDG}.
By the well known Bianchi permutability theorem for
the Christoffel and Darboux transformations,
see \cite[\S5]{Bianchi1904},
a Darboux pair of minimal surfaces yields a curved flat
in the $2$-sphere, thus a pair of hyperbolic Gauss maps
of a parallel family of flat fronts in hyperbolic space,
cf \cite[Thms 3.2 \& 3.3]{MRT},
or \cite[Sect 3]{bjr10}.

\section{Preliminaries}
\label{sec:prelim}

Our main focus will be on minimal surfaces in the Euclidean
space $\R^3$, regarded as special \emph{isothermic surfaces},
that is,
surfaces that admit (local) conformal curvature line
coordinates around every non-umbilic point.
Thus given a conformal parametrization
$$
  f:M^2\to\R^3
$$
of an isothermic surface over a Riemann surface $M^2$,
it will often be useful employ
\emph{conformal curvature line parameters}
$
  (u,v):M^2\to\R^2,
$
where the first and second fundamental forms of the
surface become particularly simple:
$$
  \langle df,df\rangle = E\,(du^2+dv^2)
   \enspace{\rm and}\enspace
  \langle df,dn\rangle = -E\,(\kappa_1du^2+\kappa_2dv^2).
$$
Note that we do not exclude umbilics,
in fact,
we will need to consider the unit sphere as an isothermic
surface, with any conformal coordinates providing a
conformal ``curvature line parametrization''.

To facilitate our analysis we consider Euclidean $3$-space
\[
  \R^3\cong\Im\H=\{ x_1i+x_2j+x_3k \,;\, x_1, x_2, x_3 \in \R\}
\]
as the set of imaginary quaternions,
thereby obtaining the Hamiltonian product
$$
  xy = -\inner{x}{y} + x \times y,
$$
of vectors as an additional algebraic structure.
However, we remark that we can replace the Hamiltonian product
by the Clifford product with only minor adjustments,
see \cite{Burstall2006} or \cite[Chapter 8]{IMDG};
see also Remark \ref{rem:CliffordAlg}.

A key argument of our analysis will be the conformal invariance
of the notion of an isothermic surface:
in particular, if $f:M^2\to\R^3$ is isothermic then so is any
of its \emph{M\"obius transforms}
$$
  \mu\circ f:M^2\to\R^3, \enspace
  \mu(x) = (ax+b)(cx+d)^{-1}.
$$

%%%%%%%%%%%%%%%%%%%%%%%%%%%%%%%%%
\subsection{Christoffel's transformation}

Given an isothermic surface $f: M^2 \rightarrow \R^3$
there is locally a second surface
$f^*: M^2 \rightarrow \R^3$
such that
(cf.\ \cite{Christoffel1867} and \cite[\S5.2.1]{IMDG}):
\begin{itemize}
\item
 $f$ and $f^*$ have parallel tangent planes at corresponding
 points, $df(T_pM^2)=df^*(T_pM^2)$ for each $p\in M^2$;
\item
 the shape operators $A$, $A^*$ of $f$ and $f^*$ commute
 (i.e., the curvature directions 
 coincide in $T_pM^2$; cf.\ \cite[Lemma 3.1.6]{IMDG}),
 and $\det A^*\det A\leq0$;
\item
 $f$ and $f^\ast$ induce conformally equivalent metrics.
\end{itemize}
Such a surface $C(f):=f^*$ is called
a \emph{Christoffel transform}
or \emph{Christoffel dual\/}
of $f$.
Away from umbilics,
a Christoffel dual is unique up to homothety:
given conformal curvature line coordinates $(u,v)$ for $f$,
a Christoffel dual $f^*$ is obtained by integrating
\emph{Christoffel's formula}
\[
  df^* = \frac2{E(\kappa_1-\kappa_2)}(dn+H\,df)
  = -\frac{1}{E}(f_udu-f_vdv).
\]
Alternatively, $f^*$ may be characterized by the fact
that
$$
  q := df\,df^* = (du^2-dv^2) + 2n\,dudv
$$
defines a holomorphic quadratic differential when
identifying the trivial normal bundle of $f$ in $\H$
with $\C$ via $n\mapsto i$.
This demonstrates the symmetry of the Christoffel duality,
that is, the fact that it can reasonably be called a
``duality'',
cf \cite[Thm 4]{hjp97} and \cite[\S5.2.6]{IMDG}:
$$
  f^{**} = f.
$$

As the Gauss map $n:M^2\to S^2$ of a minimal surface
$f:M^2\to\R^3$ is conformal,
it yields a Christoffel dual of $f$ as an isothermic surface,
$f^*=n$.
In fact, minimal surfaces in Euclidean space
can be \emph{characterized}
as isothermic surfaces whose Gauss map yields
a Chrstoffel dual,
see \cite{hjp97}, \cite[\S5.2.10]{IMDG}.
In order to ``reverse engineer'' a minimal surface $f=n^*$
from its Gauss map $n$ via Christoffel duality,
more input data is required:
prescribing a holomorphic quadratic differential $q$ on $M^2$
fixes a ``curvature line net'' on the spherical patch
$n(M^2)\subset S^2$,
thus fixes a notion of conformal ``curvature line parameters''
that allow to use Christoffel's formula to uniquely determine
a Christoffel dual
$$
  f = n^* = -\int\frac{n_udu-n_vdv}{E}
$$
from $n$.
Alternatively, other holomorphic data can be prescribed
on the Riemann surface $M^2$
to obtain the holomorphic Hopf differential of the
minimal surface $f$ and hence the minimal surface via
Christoffel's formula,
cf.\ \cite[\S5.2.11 and \S5.3.12]{IMDG}.

%%%%%%%%%%%%%%%%%%%%%%%%%%%%%%%%%
\subsection{Goursat's transformation}

Describing a minimal surface $f=\Re\Psi$ as the real part
of a holomorphic null curve $\Psi:M^2\to\C^3$ it is clear
that any orthogonal transformation $A\in O(3,\C)$ yields
a new minimal surface via
$$
  \tilde f := \Re\tilde\Psi
   \enspace{\rm with}\enspace
  \tilde\Psi := A\Psi,
$$
a \emph{Goursat transform} of $f$,
cf.\ \cite{Goursat1887},
\cite[\S83]{Nitsche}.
It is well know that the effect of such a Goursat transform
on the Gauss map is a M\"obius transformation,
thereby suggesting a extension of the notion
to isothermic surfaces,
see \cite{hj97} or \cite[Def 5.3.3]{IMDG}:
given an isothermic surface $f$
and a M\"obius transformation $\mu$ of $\R^3\cup\{\infty\}$,
the (locally defined) surface
\[
  \tilde{f}:=(\mu \circ f^*)^*
\]
will be called a \emph{Goursat transform} of $f$.
We also write $G_{\mu}(f):=\tilde{f}$.
In particular, for $\mu(x)=(x-z)^{-1}$ with $z\in\Im\H$,
we have (cf.\ \cite[Lemma 5.3.10]{IMDG})
\begin{equation}\label{eq:Goursat}
  d\!\tilde{f}=-(f^*-z)\,df(f^*-z).
\end{equation}

\subsection{Darboux's transformation}

A pair of surfaces $f,\hat{f}:M^2\to\R^3$
is called a \emph{Ribaucour pair} if
\begin{itemize}
\item
 for each $p\in M^2$, there is a $2$-sphere $s(p)$
 that has first order contact with both $f$ at $f(p)$
 and $\hat{f}$ at $\hat f(p)$;
 at $p$, and
\item
 the shape operators of $f$ and $\hat{f}$ commute;
\end{itemize}
and a Ribaucour pair is called a \emph{Darboux pair}
(cf.\ \cite{Darboux1899})
if, additionally,
\begin{itemize}
\item
 their induced metrics are conformally equivalent.
\end{itemize}
Given an isothermic surface $f:M^2\to\R^3$
and a parameter $t\in\R\setminus\{0\}$,
a \emph{Darboux transform} $\hat f$ of $f$ is obtained
by integrating the following completely integrable
Riccati equation
(cf.\ \cite{hjp97} or \cite[\S5.4.8f]{IMDG})
\begin{equation}\label{eq:Riccati1}
  d\!\hat{f} = t\,(\hat{f}-f)\,df^*(\hat{f}-f).
\end{equation}
As a consequence, an isothermic surface admits,
locally, a $(1+3)$-parameter family of Darboux transforms,
depending on the ``spectral parameter'' $t\in\R$ and
the choice of an initial point of the transform.
We note that the Darboux transformation is invariant
under M\"{o}bius transformations,
that is,
if $\mu$ is a M\"{o}bius transformation then
\begin{equation}\label{eq:mu_D}
  \mu(D_{t}(f)) = D_{t}(\mu(f)).
\end{equation}

Note that the Riccati equation \eqref{eq:Riccati1} breaks
the symmetry of the definition of the Darboux transformation:
as the Darboux transformation is involutive,
$\hat f=D_t(f)$ iff $f=D_t(\hat f)$,
we also have a corresponding Riccati equation
\begin{equation}\label{eq:Riccati_d}
  df = t\,(\hat{f}-f)\,d\!\hat{f}^*(\hat{f}-f).
\end{equation}

\subsection{Bianchi's permutability theorem}

The Christoffel and Darboux transformations commute,
that is, the following diagram commutes,
see \cite[\S5]{Bianchi1904} or \cite[Thm 5.6.3]{IMDG}:
\par
\begin{center}
  \resizebox{4.1cm}{!}{\includegraphics{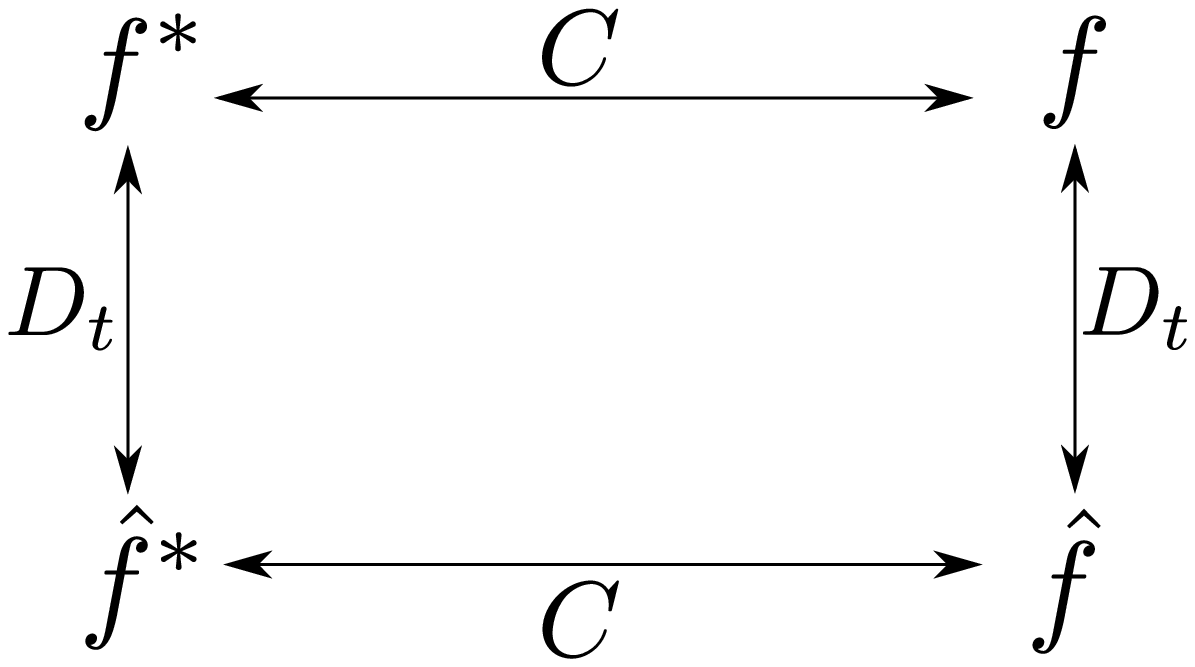}}\\
\end{center}
\noindent
More precisely, given a Christoffel dual $f^*$ and
a Darboux transform $\hat f=D_t(f)$ of an isothermic
surface $f:M^2\to\R^3=\Im\H$,
the surface
\begin{equation}\label{eq:Riccati_ori}
  \hat{f}^*:= f^* + \frac1{t} (\hat{f}-f)^{-1}
\end{equation}
is a Christoffel dual of $\hat{f}$
and a $D_{t}$-transform of $f^*$ at the same time.

Note that this permutability scheme occurs as a limiting
case of the commonly known Bianchi permutability scheme
of \cite[\S3]{Bianchi1904} (cf.\ \cite[\S5.6.6]{IMDG})
for the Darboux transformations when one of the
spectral parameters vanishes.

\section{Permutability theorem}
\label{sec:permutability}

We are now prepared to state and prove our main
permutability theorem for the Christoffel,
Goursat and Darboux transformations:

\begin{theorem}\label{thm:main}
Given
a Darboux transform $\hat f=D_t(f)$ and
a Goursat transform $\tilde f=G_\mu(f)$
of an isothermic surface $f:M^2\to\R^3$,
the surface
\begin{equation}\label{eq:GD}
  G_{\mu}(\hat{f}) :
  = G_{\mu}(f) + \frac{1}{t}\,
    \left(\mu\circ\hat{f}^*-\mu\circ f^*\right)^{-1}
\end{equation}
is simultaneously a Goursat transform of $\hat{f}$
and a $D_{t}$-transform of $\tilde{f}$.
Moreover the following diagram commutes:
\par
\begin{center}
  \resizebox{8cm}{!}{\includegraphics{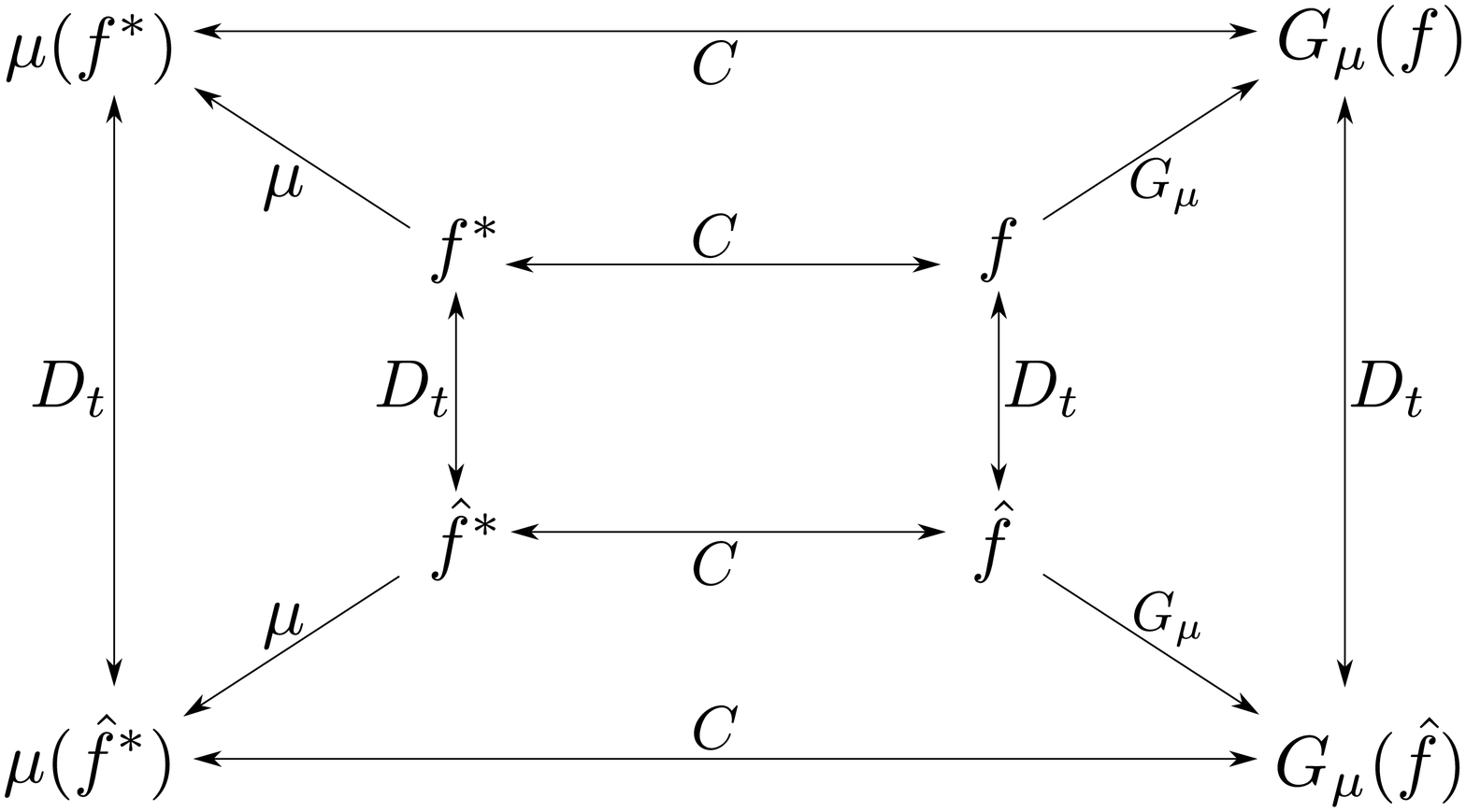}}\\
\end{center}
\end{theorem}

\begin{proof}
Let $f$ be isothermic and $G_\mu(f)$ and $D_t(f)$
a Goursat resp.\ Darboux transform of $f$.
The Goursat transform $G_\mu(f)$ is given by intertwining
the Christoffel duality with a M\"obius transformation,
$$
  G_\mu(f) = (\mu\circ f^*)^*.
$$
Thus the upper quadrilateral of our diagram reflects
the definition of the Goursat transformation.
Then, by Bianchi's permutability theorem,
\eqref{eq:Riccati_ori} can serve to complete
the inner quadrilateral of the diagram,
$$
  \hat{f}^* = f^* + \frac1{t} (\hat{f}-f)^{-1}.
$$
Now the M\"{o}bius invariance of $D_t$ (cf.\ \eqref{eq:mu_D})
yields the left quadrilateral in the diagram,
\[
  D_t(\mu(f^*))=\mu(D_t(f^*))=\mu(\hat{f}^*),
\]
and another application of Bianchi's permutability theorem
serves to obtain
$$
  f^\dagger := G_{\mu}(f)
   + \frac{1}{t}\,(\mu\circ\hat{f}^*-\mu\circ f^*)^{-1}
$$
as a Darboux transform of $G_\mu(f)$
and a Christoffel dual of $\mu\circ\hat f^*$
at the same time,
$$
  f^\dagger
  = D_t(G_\mu(h))
  = (\mu\circ\hat f^*)^*
  = G_\mu(\hat f),
$$
where the last equality is due to the definition
of the Goursat transformation, again.
\end{proof}

\begin{remark}
\label{rem:CliffordAlg}
Although we used the Hamiltonian product on $\R^3=\Im\H$
thought of as the space of imaginary quaternions,
all considerations can equally well be formulated using
the Clifford product on $\R^3$.
This fact hinges on how the symmetric parts of both,
the Hamiltonian and Clifford products,
are related to the inner product of $\R^3$,
so that the geometric content of equations such as
the Riccati equation \eqref{eq:Riccati1} can equivalently
be expressed in terms of either product,
cf.\ \cite[Sect 2]{Burstall2006} or \cite[Chap 8]{IMDG}.
However, the Clifford product can be employed in $\R^n$,
hence Theorem \ref{thm:main} holds in arbitrary codimension.
\end{remark}

\section{Weierstrass data of minimal Darboux transformations}
\label{sec:mDarboux}

Now we focus on the particular M\"obius transformation,
$$
  \iota:\R^3\cup\{\infty\}\to\R^3\cup\{\infty\}, \enspace
  x\mapsto\iota(x) := -i - 2(i+x)^{-1},
$$
which yields
the stereographic projection and its inverse
when restricted to the $2$-sphere $S^2\subset\R^3$
and $\C=\{i\}^\perp\subset\R^3$, respectively:
$$
  \iota|_{S^2} = \sigma
   \enspace{\rm and}\enspace
  \iota|_{\C j} = \sigma^{-1}.
$$
From two holomorphic (or meromorphic) functions $g,h:M^2\to\C$
on the Riemann surface $M^2$ we obtain a degenerate
Christoffel pair
\[
  h\,j,\qquad
  (hj)^* = -j\,g,
\]
cf.\ \cite[\S5.2.11]{IMDG}.
Following \cite[\S5.3.12]{IMDG}
we consider the Goursat transformation $f=G_\iota(hj)$
that is obtained as a Christoffel transform of the
(holomorphic) Gauss map
\begin{equation}\label{eq:nu}
  n = \mu\circ(hj)^*
  = \frac{1}{1+|g|^2}
    \left((1-|g|^2)\,i - 2j\,g \right) : M^2 \to S^2,
\end{equation}
by integrating
\begin{equation}\label{eq:Goursat-repr}
  df
  = 2\Re(g\omega)\,i
  + \Re((1-g^2)\omega)\,j
  + \Im((1+g^2)\omega)\,k
\end{equation}
with $\omega:=\frac{1}{2}dh$,
as obtained from \eqref{eq:Goursat}.
This identifies the Weierstrass representation
for minimal surfaces
as a special case of the Goursat transformation
for isothermic surfaces,
as annunciated above.

Thus we consider the \emph{Weierstrass data}
of the minimal surface $f$ to be one of the following
equivalent sets of data:
\begin{itemize}
\item
 the pair $(g,h)$ of the meromorphic resp.\ holomorphic
 functions $g$ and $h$;
\item
 the pair $(g,\omega)$ of the meromorphic function $g$
 and a suitable holomorphic differential $\omega$; or
\item
 the pair $(g,q)$ of the meromorphic function $g$
 and a holomorphic quadratic differential $q$
 (i.e., the ``polarization'' of $M^2$).
\end{itemize}
Here the second items of these data sets are related
by the equations
$$
  2\omega\,dg = dh\,dg = q.
$$
Note that, for a minimal surface $f$ with its Gauss map
$n$ as a Christoffel transform, our holomorphic quadratic
differential is related to its Hopf differential by
$$
  q = dh\,dg
  \simeq df\,df^*
  = 2\,(\Re\langle f_{zz},n\rangle\,dz^2
  +  n\,\Im\langle f_{zz},n\rangle\,dz^2).
$$

\subsection{Application of Theorem \ref{thm:main}}

As an application of Theorem \ref{thm:main}
we can now derive a formula that relates the Weierstrass data
of the surfaces of a minimal Darboux pair,
cf \cite[Cor 3.4]{MRT}:

\begin{corollary}
\label{fact:to_prove}
Let $M^2$ be a Riemann surface and $f,\hat{f}:M^2\to\R^3$
non-flat conformal minimal immersions,
given by Weierstrass data $(g,q)$ and $(\hat{g},\hat{q})$,
respectively.
Then $(f,\hat{f})$ is, up to translation, a Darboux pair
if and only if there is $t\in\R\setminus\{0\}$ so that
\begin{equation}\label{eq:holomorphic_Riccati}
  dg\,d\hat g = t(\hat g-g)^2\,q
   \enspace{and}\enspace
  \hat q = q.
\end{equation}
\end{corollary}

Here we consider $q$ and $\hat q$ as $\C$-valued holomorphic
quadratic differentials.

Note that, with $\omega:=\frac{q}{2dg}$,
the first equation of \eqref{eq:holomorphic_Riccati}
yields the Riccati equation
$$
  d\hat g = 2t(\hat g-g)^2\omega
$$
of \cite[(3.11)]{MRT},
while the transformation formula \cite[(3.12)]{MRT}
for the Weierstrass data is obtained using $\hat q=q$:
$$
  \hat\omega = \frac{\hat q}{2d\hat g}
  = \frac{dg}{2t(\hat g-g)^2}.
$$

A lemma will be helpful to prove the theorem,
cf \cite[(58)]{hjp97} or \cite[\S5.4.15]{IMDG}:

\begin{lemma}
\label{thm:mean_curvature}
Let $f:M^2\to\R^3$ be isothermic,
and let $f^*$ and $\hat f=D_t(f)$ denote a Christoffel
and a Darboux transform of $f$, respectively.
Then
$$
  0 = \hat H\,|\hat f-f|^2
  - 2\langle\hat f-f,n\rangle
  + \frac{1}{t}\,H^*,
$$
where $n$ denotes a Gauss map of $f$, and the Gauss maps
$n^*$ and $\hat n$ are chosen so that the holomorphic
quadratic differentials $\hat q\simeq q^*\simeq q$.
\end{lemma}

\begin{proof}
Given a Gauss map $n$ of $f$, there is a consistent choice
of Gauss maps
$$
  n^* := -n
   \enspace{\rm and}\enspace
  \hat n := -(\hat f-f)^{-1}n(\hat f-f)
$$
so that $\hat q\simeq q^*\simeq q$ when identifying
$\hat n\simeq n^*\simeq n$ as imaginary units of the
normal bundles of $\hat f$, $f^*$ and $f$ in $\H$,
as complex rank $1$ bundles:
using $f^{**}=f$,
$$
  q^* = df^*df = \overline{df\,df^*} = \bar q,
   \enspace{\rm hence}\enspace
  q^* \simeq q;
$$
and using the Riccati equations and Bianchi's permutability
for the Christoffel and Darboux transformations,
$$
  d\hat f = t(\hat f-f)df^*(\hat f-f)
   \enspace{\rm and}\enspace
  \hat f^* = f^* + \frac{1}{t}(\hat f-f)^{-1},
$$
we deduce that
$$
  \hat q
  = d\!\hat f\,d\!\hat f^*
  = (\hat f-f)df^*df(\hat f-f)^{-1}
  = (\hat f-f)\bar q(\hat f-f)^{-1},
   \enspace{\rm hence}\enspace
  \hat q \simeq q.
$$
Then direct computation yields with $n^*=-n$
$$
  d\hat n = (\hat f-f)^{-1}\{
   dn^* + 2\langle\hat f-f,n\rangle
    (t\,df^* - \frac{1}{|\hat f-f|^2}\,df)
  \}(\hat f-f),
$$
hence the second fundamental form of $\hat f$ with
$\hat n$ as a normal field computes to
$$
  \Re d\!\hat f\,d\hat n
  = -t\,|\hat f-f|^2\{ \Re df^*dn^*
    - 2t\langle\hat f-f,n\rangle\{
      t|df^*|^2+\frac{1}{|\hat f-f|^2}\,\Re q^* \} \}
$$
and taking trace with respect to
$|d\!\hat f|^2=t^2|\hat f-f|^4|df^*|^2$
yields the claim since $q^*$ is a quadratic
differential so that $\Re q^*$ is trace free.
\end{proof}

Now Corollary \ref{fact:to_prove} follows directly
from our Permutability Theorem \ref{thm:main}:

\begin{proof}
First assume that \eqref{eq:holomorphic_Riccati} holds
for some $t\in\R$,
that is,
with $d(-jg)^*=\frac{q}{dg}\,j$ we have
$$
  d(-j\hat g)
  = t\,(-j\hat g+jg)\,d(-jg)^*(-j\hat g+jg)
$$
showing that $-j\hat g=D_t(-jg)$:
hence, wheeling out our Permutability Theorem \ref{thm:main},
we learn that $f$ and $\hat f$ form a Darboux pair when
appropriately positioned in $\R^3$.

To see the converse we employ Lemma \ref{thm:mean_curvature}:
with $H^*=1$ and $\hat H=0$ this yields
$$
  \hat f^*
  = f^* + \frac{1}{t}(\hat f-f)^{-1}
  = n + 2\langle\hat f-f,n\rangle(\hat f-f)^{-1}
  = \hat n,
$$
Hence, unleashing our Permutability Theorem \ref{thm:main}
gaain we obtain the result, as above.
\end{proof}

As a direct consequence of \eqref{eq:GD} from
Thm \ref{thm:main} and Cor \ref{fact:to_prove}
we obtain:

\begin{corollary}
Let $f,\hat{f}:M^2\to\R^3$ form a minimal Darboux pair,
with Gauss maps $n$ resp $\hat n$.
Then
\[
  \hat{f} = f + \frac{1}{t}(\hat{n}-n)^{-1}.
\]
\end{corollary}

\subsection{Curved flats and flat fronts}
\label{sec:flats}

By Corollary \ref{fact:to_prove},
resp \cite[Cor 3.4]{MRT},
the Gauss maps $n$ and $\hat n$ of a minimal Darboux
pair $(f,\hat f)$ in Euclidean space form a (degenerate)
Darboux pair themselves, hence a curved flat in the space
of point pairs in the $2$-sphere,
see \cite{FerusPedit} or \cite[\S3.3.2 or \S5.5.20]{IMDG}.
Consequently, interpreting the interior of the $2$-sphere
as a Poincar\'e ball model of hyperbolic space,
$(n,\hat n)$ qualifies as the pair of hyperbolic Gauss maps
of a parallel family of flat fronts,
see \cite{KUY}.
More precisely, linearization of of the Riccati equations
(\ref{eq:holomorphic_Riccati}) for $\hat g$ reads
$$
  0 = d\left({\hat g\hat a\atop\hat a}\right)
  + 2t\,\left( {g\omega\atop\hfill\omega}\,
     {-g\omega g\atop\hfill-\omega g} \right)
    \left({\hat g\hat a\atop\hat a}\right)
  = d\left({\hat g\hat a\atop\hat a}\right)
  + t\,\left({ga\atop a}\right)\,a^{-1}\omega(\hat g-g)\hat a,
$$
where $d\hat a+2t\omega(\hat g-g)\hat a=0$,
and similary for $dg$.
Thus we obtain the curved flat system
$$
  d\left({ga\atop a}\,{\hat g\hat a\atop\hat a}\right)
  + t\,\left({ga\atop a}\,{\hat g\hat a\atop\hat a}\right)
  \left({0\atop\hat a^{-1}\hat\omega(g-\hat g)a}\,
   {a^{-1}\omega(\hat g-g)\hat a\atop 0}\right) = 0
$$
of the Weierstrass type representation for flat fronts
in hyperbolic space.

Thus we recover the relation between
minimal Darboux pairs in Euclidean space and
(parallel families of) flat fronts in hyperbolic space
of \cite[Thms 3.2 \& 3.3]{MRT},
see also \cite[Sect 3]{bjr10}:

\begin{corollary}
\label{thm:curved_flat}
Any minimal Darboux pair $(f,\hat f)$,
with $(n,\hat n)$ as its Darboux pair of Gauss maps,
determines a parallel family of flat fronts in hyperbolic
space, with hyperbolic Gauss maps $n$ and $\hat n$.\par
Conversely, the hyperbolic Gauss maps of a flat front
in hyperbolic space form a curved flat in the $2$-sphere,
hence determine a minimal Darboux pair
in Euclidean space.
\end{corollary}

Note that the above linearization of the Riccati equations
can equally be performed for maps into $\R^3$ or $\R^n$
by using quaternions or an appropriate Clifford algebra,
see \cite[\S5.4.1 and \S8.7.2]{IMDG}
or \cite[Sect 3.2]{Burstall2006}.

\begin{acknowledgements}
The second author expresses his gratitude to the members
of the Institute of Discrete Mathematics and Geometry
at TU Wien for their hospitality during his stay in 2015/16. 
\end{acknowledgements}


\begin{thebibliography}{99}

\bibitem{Bianchi1904}
  L.~Bianchi,
  \emph{Il teorema di permutabilit\`a per le trasformazioni
  di Darboux delle superficie isoterme\/},
  Rom. Acc. L. Rend. \textbf{13} (1904), 359--367.

\bibitem{bjpp97}
 F.E.~Burstall, U.~Hertrich-Jeromin, F.~Pedit, U.~Pinkall,
 \emph{Curved flats and isothermic surfaces},
 Math.\ Z.\ 225 (1997) 199--209

\bibitem{Burstall2006}
 F.E.~Burstall,
 \emph{Isothermic surfaces: conformal geometry, {C}lifford
  algebras and integrable systems},
 Integrable systems, geometry, and topology,
 AMS/IP Stud. Adv. Math., vol.~36, Amer. Math. Soc.,
 Providence, RI, 2006, pp.~1--82.

\bibitem{bjr10}
 F.E.~Burstall, I.~Hertrich-Jeromin, W.~Rossman:
 \emph{Lie geometry of flat fronts in hyperbolic space},
 C.\ R.\ 348 (2010) 661--664

\bibitem{Christoffel1867}
  E.~Christoffel,
  \emph{Ueber einige allgemeine Eigenschaften der
   Minimumsfl\"{a}chen\/},
  J.\ Reine Angew.\ Math.\ \textbf{67} (1867), 218--228.

\bibitem{Darboux1899}
  G.~Darboux,
  \emph{Sur les surfaces isothermiques\/},
  C.R.\ Acad.\ Sci.\ Paris \textbf{128} (1899), 1299--1305, 1538.

\bibitem{FerusPedit}
  D.~Ferus and F.~Pedit,
  \emph{Curved flats in symmetric spaces\/},
  Manuscripta Math.\ \textbf{91} (1996), 445--454.

\bibitem{Goursat1887}
  E.~Goursat,
  \emph{Sur un mode de transformation des surfaces minima\/},
  Acta Math.\ \textbf{11} (1887), 135--186, 257--264.

\bibitem{hj97}
  U.~Hertrich-Jeromin,
  \emph{Supplement on curved flats in the space of point pairs
   and isothermic surfaces: a quaternionic calculus},
  Doc.\ Math.\ \textbf{2} (1997), 335--350.

\bibitem{hjp97}
  U.~Hertrich-Jeromin and F.~Pedit,
  \emph{Remarks on the {D}arboux transform of isothermic surfaces},
  Doc.\ Math.\ \textbf{2} (1997), 313--333.

\bibitem{IMDG}
  U.~Hertrich-Jeromin,
  \emph{Introduction to {M}\"obius differential geometry},
  London Mathematical Society Lecture Note Series, vol. 300, Cambridge
  University Press, Cambridge, 2003.

\bibitem{KUY}
  M.~Kokubu, M.~Umehara and K.~Yamada,
  \emph{Flat fronts in hyperbolic 3-space},
  Pacific J.\ Math.\ \textbf{216} (2004), 149--175.

\bibitem{MRT}
  A.~Mart{\'{\i}}nez, P.~Roitman, and K.~Tenenblat,
  \emph{A connection between flat fronts in hyperbolic space and
  minimal surfaces in euclidean space},
  Ann.\ Global Anal.\ Geom.\ \textbf{48} (2015), 233--254.

\bibitem{Nitsche}
  J.C.C.~Nitsche,
  \emph{Lectures on minimal surfaces. {V}ol. 1},
  Cambridge University Press, Cambridge, 1989.

\end{thebibliography}
\end{document}